\documentclass[11pt]{article}

\usepackage{amsmath,amsthm,amscd,amsfonts,amssymb,enumerate}
\usepackage{url}

\newtheorem{theorem}{Theorem}[section]

\newtheorem{corollary}[theorem]{Corollary}
\newtheorem{definition}[theorem]{Definition}

\numberwithin{equation}{section} 
\theoremstyle{definition}
\newtheorem{remark}[theorem]{Remark}
\newtheorem{problem}{Problem}

\newcommand{\kk}{\mathbf{k}}
\newcommand{\xx}{\mathbf{x}}
\newcommand{\yy}{\mathbf{y}}
\newcommand{\zz}{\mathbf{z}}
\newcommand{\ww}{\mathbf{w}}

\title{Beyond Shapiro's problem: \\ from cyclic sums to ``graphic'' sums}
     \author{Sergey Sadov\footnotemark[1]\footnote{e-mail: serge.sadov@gmail.com}} 
\date{}

\begin{document}

\maketitle

\begin{abstract}
A review of author's work on cyclic inequalities of Shapiro-Diananda type and related optimization
problems is presented.


\end{abstract}

\section{Introduction}
The purpose of this communication is to review some nonstandard inequalities and asymptotics obtained by the author in recent years and to explain the line of thought that led to them. This is a technical paper as much as a personal story; for this reason ``we hope that the author will be forgiven'' for writing in the first person.

If a short keyword is wanted to characterize the differently looking inequalities I intend to discuss, a good candidate is {\em quasi-cyclic}. 
While an attempt to devise a formal framework to accommodate everything here does not seem worthwhile, some common formal features can be pointed out:

\begin{itemize}
\item The functions being estimated are {\em homogeneous of order zero}. (Or a concerned inequality can be transformed into a form involving such functions).

\item The functions depend on $n$ variables, where $n$ is a parameter, although in some examples it will be a small constant, e.g.\ $n=3$. The function is defined as a {\em sum}\ of terms whose number grows with $n$ (often just equals $n$).

\item A typical inquiry concerns the {\em asymptotic behaviour}\ of the best estimate for $n$-th sum of the 
given type for large $n$.

\item The {\em term} is a fraction whose numerator and denominator are increasing functions of the variables upon which they depend; each of the $n$ variables appears at least once in some numerator and in some denominator. So the effect of changing any one variable
is not obvious.

\item There is a {\em pattern}\ prescribing the specific
form of the terms in a sum of the given type. 
A pattern can be {\em local}, e.g. $x_1/(x_2+x_3)$ is a pattern for Shapiro's sums, or it can involve potentially any number of variables, as will be seen in further examples.

\item The structure of the sum (indices of variables involved in the terms) is governed either by {\em symmetry} --- perfect (cyclic shift) or approximate (index shift and some boundary conditions), --- or in a combinatorial way, by a {\em graph}. 
\end{itemize}

This paper is neither meant to be a comprehensive survey of cyclic inequalities 
(for that consult \cite[Ch.~XVI]{Mitrinovic_1993}) or particularly Shapiro's and Shapiro-type inequalities 
(see e.g.~\cite{Clausing_1992}, \cite{S21b}), nor
a full compendium of the author's results in the papers
\cite{S16}--\cite{S22b}, \cite{KalachevSadov_2018}. Rather it is meant to give a glimplse of the inequalities and asymptotics that came into being over time --- by a deliberate mind wandering if not quite by chance --- after I learned about Shapiro's problem and its fascinating solution by Drinfeld. 

The focus will be not as much on the various {\em results} as on the motivating {\em questions}.

The values of the indeterminates are assumed to be nonnegative throughout and the denominators must be positive. 

First thing first, the foundation of all that follows is the celebrated AM-GM inequality between the arithmetic and geometric means. In the straightforward,
name-befitting form it reads
$$
 \frac{a_1+\dots+\dots+a_n}{n}\geq(a_1\cdot\dots \cdot a_n)^{1/n}
$$
but there are multiple other ways to present it, which will be relevant in the sequel.

Putting $p=a_1\cdot\dots\cdot a_n$, upon the substitution $a_i=p^{1/n}x_i/x_{i+1}$, we come to the AM-GM in the homogeneous form
$$
 \frac{x_1}{x_2}+\frac{x_2}{x_3}+\dots+\frac{x_n}{x_1}\geq n,
$$
which is the simplest, purest and the most important of all cyclic inequalities. The first summand $x_1/x_2$ 
determines the local pattern that repeats itself, with index shift, until the closing of the cycle.

One can also ``uncycle'' the latter cyclic problem
and ask, given a parameter $x$, what is the minimum value of the sum
$
x_1/x_2+\dots+x_{n}/{x_{n+1}}
$
assuming that $x_1=x$ and $x_{n+1}=1$ (say). This is a constrained optimization problem; the objective function involves the same local pattern as before and the cyclic condition $x_{n+1}=x_1$ is replaced by the boundary conditions containing a parameter. It was such a relief for me to realize that the AM-GM can be treated regularly and ``thoughtlessly'' by the Lagrange multipliers method after a fearful Cauchy's up-down inductive proof learned --- or, more precisely, presented to us --- in school!

So far, the ordering of the variables $x_1,\dots,x_n$
seemed to be of importance. However, one can also write the AM-GM in the form
\begin{equation}
\label{AMGMperm}
 \sum_{i=1}^n \frac{x_i}{x_{\sigma(i)}}\geq n,
\end{equation}
where $\sigma$ is any bijection of the set $[n]=\{1,\dots,n\}$ onto itself, in other words, a permutation of the index set. The permutation is always
a union of cycles, and there can be more than once cycle in $\sigma$. Formally, the inequality in this form
for a particular $n$ depends on the standard AM-GM for 
all cycle lengths $m\leq n$.

Quite another point of view at the AM-GM, providing an additional insight, is the dynamic programming form, due to Bellman:
determine, in a closed form if possible, the functions $f_n(x)$ defined by the recurrence
\begin{equation}
\label{AMGMfeq}
 f_{n}(x)=\max_{y>0} \left(\frac{x}{y}+f_{n-1}(y)\right).
\end{equation}
The answer is, of course, $f(x)=nx^{1/n}$.

For a quick look into what lies ahead, here are a few samples.
These numbered problems will be roadposts on our tour. 
Some are easy, some quite difficult,
% (one might be unsolved!) 
but without the AM-GM they all would be hopeless. 

\begin{problem}
\label{prob:Nesbitt} Prove the inequality
$$
 \frac{a}{b+c}+\frac{b}{c+a}+\frac{c}{a+b}\geq\frac{3}{2}.
$$
\end{problem} 

\begin{problem}
\label{prob:Shapiro}
Is it true that for any $n$
$$
 \frac{x_1}{x_2+x_3}+\frac{x_2}{x_3+x_4}+\dots+\frac{x_{n-1}}{x_n+x_1}+\frac{x_n}{x_1+x_2}\geq \frac{n}{2}?
$$
\end{problem} 
  
\begin{problem}
\label{prob:minsum}
Show that
$$
\min\left(\frac{x_1}{x_3}+\frac{x_2}{x_3}+
\frac{x_3}{\min(x_1,x_2)}\right)=2\sqrt{2}.
$$
\end{problem}

\begin{problem}
\label{prob:extrminsum}
{\em Number game}: There are 40 people sitting in the room. %Each one has at least one friend. 
Any two, say, $A$ and $B$,
are connected through a chain of friends: $B$ is a friend of (a friend of \dots) $A$.  
Friendship is not necessarily a symmetric relationship. (It is supposed that there is a chain from $A$ to $B$ and also from $B$ to $A$, not necessarily through the same people.)

Everyone thinks of a positive real number. Then everyone asks his/her friends of their numbers, takes their minimum, and divides own number by that minimum. The quotient is reported to the game's host.
Finally, the host sums up all the 40 reported quotients.
Question: can the sum ever be less than $9.8$?
\end{problem} 

\begin{problem}
\label{prob:maxsum}
Find the greatest lower bound of the sum
$$
 \frac{x_1}{x_2}+\frac{x_2}{\max(x_3,x_4)}+\frac{x_3}{\max(x_4,x_5)}
+\frac{x_4}{\max(x_1,x_6)}+\frac{x_5}{\max(x_1,x_2,x_4)}+
\frac{x_6}{\max(x_3,x_5)}.
$$
\end{problem} 

\begin{problem}
\label{prob:extrmaxsum}
{\em Number game}: The same setting as in Problem~\ref{prob:minsum}, but now everyone has exactly 12 friends and, instead of the minimum, he or she takes
the maximum of friends' values.
Again, the host sums up all the 40 reported quotients.
%Question: 
Can the sum be less than $4.9$?
\end{problem} 

\begin{problem}
\label{prob:Shallit}
Let $\xx=(x_1,\dots,x_n)$ be a vector with positive components.
Denote
%$$
% f_n(\xx)=\sum_{i=1}^n x_i+\sum_{1\le i\le j\le n} \prod_{k=i}^j\frac{1}{x_k}.
%$$
$$
 g_n(\xx)=\sum_{j=1}^n \left(x_j+\frac{1}{x_j}\right)+\sum_{j=1}^{n-1}\frac{x_j}{x_{j+1}}.
$$
Show that there exists a positive constant $C$ such that
\begin{equation}
\label{ShAsym}
 \min_{\xx>0} g_n(\xx)=3n-C+o(1) 
 \qquad (n\to\infty).
\nonumber
\end{equation}
\end{problem} 

\begin{problem}
\label{prob:Mavlo} Prove a generalization of the inequality of Problem~\ref{prob:Nesbitt}:
for any $x>0$,
$$
 \frac{a}{b+cx}+\frac{b}{c+ax}+\frac{c}{a+bx}\geq\frac{3x}{1+x^3}.
$$
\end{problem} 

\begin{problem}
\label{prob:maxforward}
% Story "People in the circle"
{\em Number game}: 
There are $2022$ people sitting in a circle. Everyone thinks of a positive real number. Let $x_i$ be the number chosen by the person sitting at the $i$-th place. Then everyone rolls a 2022-faced dice and gets some natural number $r_i\leq 2022$. He or she asks his/her $r_i$ immediate clockwise neighbors of their numbers and calculates their arithmetic mean; then reports $x_i$ divided by this mean to the game's host. Finally, the host sums up all the 2022 reported quotients.    
Question: can the result ever be less than $22$?
\end{problem} 

\begin{problem}
\label{prob:modifiedAMGM}
Prove the inequality
$$
 a+\frac{b}{a+1}+\frac{c}{b+1}+\frac{25/2}{c+1}\geq\frac{11}{2}.
% 2a+\frac{2b}{a+1}+\frac{2c}{b+1}+\frac{25}{c+1}\geq 11.
$$
\end{problem} 

\section{Cyclic sums of Shapiro and Diananda}    

Problem~\ref{prob:Nesbitt} was pubilshed by Nesbitt in 1903 \cite{Nesbitt_1903}
according to \cite{Diananda_1977} who gives three proofs (reproduced in \cite{Mitrinovic_1993}). It has become popular as one of the ``what every math olympiad participant must know''. 

Problem~\ref{prob:Shapiro} due to Harold S.~Shapiro
is a natural generalization and the actual beginning of the ``theory'' of cyclic inequalities.
It had to take some courage from the proposer in 1954 to admit that he was unable to answer the seemingly elementary question ``obviously'' in the affirmative. 
The involvement of some big names, among them the number theorist L.\,J~Mordell, and a discovery that the answer was actually negative turned the innocently looking problem into a serious challenge: to find the true lower bound of Shapiro's cyclic sum divided by $n/2$. 

The origin of the difficulty lies in the fact that the shift-invariant $n$-tuple $(1,\dots,1)$ is not a minimizer for Shapiro's sum for even $n\geq 14$ and odd $n\geq 25$. (The case $n=23$ was solved numerically only in 1989 \cite{Troesch_1989}.)
Thus the symmetry of the problem does not transfer to
the symmetry of an extremizer. 

The analytical solution of Shapiro's problem with best constant was found by V.\,G.~Drinfeld in 1969 (when he was a ten-grader working under supervision of Prof.~V.L.~Levin). In it, the rearrangement (Chebyshev's)
inequality played the crusial role, along, of course, with
AM-GM. We refer to \cite{Clausing_1992} or \cite{Mitrinovic_1993} for details. 

It is natural to inquire about the best lower bound if one uses a pattern different from $x_1/(x_2+x_3)$ that reproduces itself with index shift in Shapiro's sum.  
Diananda \cite{Diananda_1959} proposed a choise that leads not to a single pattern but to a family of patterns: we just change the length of the denominators in Shapiro's sum. Call the expressions 
$$
 S_{n,k}(\xx)=k\sum_{j=1}^n \frac{x_j}{x_{j+1}+\dots+x_{j+k}}
$$
the Diananda sums. Here, as in Shapiro's case, the index $j+k$ is understood modulo $n$. The inclusion of the normalizing factor $k$ in the definition will be justified in Sec.~\ref{sec:psums} and from a different point of view in Sec.~\ref{sec:maxforward}. (The left-hand side of Shapiro's inequality is $\frac{1}{2}S_{n,2}$.)
Put 
$$
A_{n,k}=\min_{\xx} S_{n,k}(\xx).
$$
Obviously, $A_{n,k}\leq n$, since taking
all $x_j=1$ makes the sum equal $n$.   
Diananda found a sufficient condition under which the minimum is equal to $n$. For a fixed $k$, it yields 
a finite number of values of $n$. (Note that \cite[Theorem~3]{BoarderDaykin_1973} quote Diananda's theorem \cite[Theorem~1]{Diananda_1959} incorrectly.)

Diananda \cite{Diananda_1961} also showed that $n^{-1}A_{n,k}\geq k^{-1}$
for all $k\geq 1$ and $n\geq k$. His overly simple argument led me to doubt that $k^{-1}$ was the true asymptotic order of $\inf_n n^{-1}A_{n,k}$ as $k\to \infty$.  

This prompted the research \cite{S16} where I proved

\begin{theorem}
\label{thm:diananda}
For all $n\geq k\geq 1$ the inequality
$$
 k(2^{1/k}-1)\leq \frac{A_{n,k}}{n}\leq \gamma_k
$$
holds true, where $\gamma_k$ are roots of certain transcendental equations. They form a decreasing sequence;
$\gamma_2\approx 0.98913$ is Drinfeld's constant for Shapiro's problem and $\lim_{k\to\infty}\gamma_k\approx 0.930498$.
\end{theorem}

Note that $\lim_{k\to\infty} k(2^{1/k}-1)=\ln 2\approx 0.693$, so there is a considerable gap between the upper and lower bounds. For $k=2$ the lower bound coincides
with $\gamma_2$ by Drinfeld's proof and there is no gap.

For $k\geq 3$ Drinfeld's method does not seem to be applicable for obtaining the lower bound;
in the above theorem it is obtained by a different (still elementary calculus) method.

The situation about the gap is in fact a lot worse.
There is no known {\em numerical}\ method to approximate the true value of the constant $\inf_n n^{-1}A_{n,k}$
if $k\geq 3$. The old computation by Boarder and Daykin \cite{BoarderDaykin_1973}, quoted in  \cite[Eq.~(27.41)]{Mitrinovic_1993}, shows that for $k=3$ this constant does not exceed $0.97794$, which is consistent with numerical value of $\gamma_3\approx 0.97793$. But it is unclear how, in principle, one might numerically obtain a lower bound better than $3\cdot (2^{1/3}-1)\approx 0.77976)$ given by Theorem~\ref{thm:diananda}.

Unable to close the mentioned gap, I followed P\'{o}lya's advice ``if you cannot solve the given problem, try to find a similar one that you can solve''. Variants abound
and the exposition forks here. First I will describe a path leading to the notion of graphic sums and optimization problems related to them. This will be a rather long story. An independent short thread forking from here opens in Sec.~\ref{sec:maxforward}.
%, where I discuss a generalization of the Diananda sums in a different direction. 

\section{The Shapiro-Diananda $p$-sums}    
\label{sec:psums}

Consider the power means $M_{k,p}(x_1,\dots,x_k)=[(x_1^p+\dots+x_k^p)/k]^{1/p}$.   

\begin{definition}
The {\em Shapiro-Diananda cyclic $p$-sums}\ are the functions of $n$ nonnegative variables
\begin{align}
  S_{n,k,p}(x_1,\dots,x_n)=\sum_{j=1}^n \frac{x_j}{M_{k,p}(x_{j+1},\dots,x_{j+k})}.
\label{Snkp}
\end{align}	
As before, the indices are treated modulo $n$.
\end{definition} 

For $p=1$, these sums coincide with Diananda's sums $S_{n,k}$ introduced earlier. With the chosen normalization we have the monotonicity property:
$S_{n,k,p}(\xx)>S_{n,k,p'}(\xx)$ whenever $p<p'$,
due to the monotonicity of power means \cite[\S~2.9]{HardyLittlewoodPolya1934}.

Denote 
$$
 A_{n,k,p}=\inf_{\xx} S_{n,k,p}(\xx)
$$
and 
$$
 B_{k,p}=\inf_{n\geq 1}\frac{A_{n,k,p}}{n}.
$$

Choosing $\xx=(1,\dots,1)$ we see that 
$A_{n,k,p}\leq n$ for any $n$, $k$ and $p$; hence
$B_{k,p}\leq 1$ for any $k$ and $p$.

This section does not go far in depth. It rather 
prepares an excuse for a further generalization.
More could be said here if what looks doable
were done, %yet --- to the best of my knowledge --- has not been done, 
which is the following:

\begin{enumerate}[(i)]
\item 
Find the analog of Drinfeld's result for $k=2$ and any $p$
namely, give a constructive analytic description of the constant $B_{2,p}$ in terms of some transcendental equation; explore the properties of the function $p\mapsto B_{2,p}$. 

\item 
Establish an analog of Theorem~\ref{thm:diananda} for general values of $p$ (perhaps separately for $p>0$,
$p<0$ and in the special case $p=0$). Explore the behaviour of the obtained bounds as functions of $p$.
\end{enumerate}

The rationale for (i) is that Drinfeld's method works well for a variety of patterns involving the variables $x_1$, $x_2$, $x_3$, see e.g.\ a review of results by Godunova and Levin in \cite{Mitrinovic_1993}.

As concerns (ii), let me present one simple case.

\begin{theorem}
\label{thm:p-1}
For any $k\geq 1$, the equality $B_{k,-1}=1$ holds true.
\end{theorem}

\begin{proof}
It is easy to see that
$$
 S_{n,k,-1}(\xx)=\frac{1}{k}\sum_{j=1}^n\frac{y_{j+1}+\dots+y_{j+k}}{y_j},
$$
where $y_j=x_j^{-1}$.
The sum in the right-hand side consists of $k$ parts, each of which can be written as
$\sum_{j=1}^n y_{\sigma(j)}/{y_j}$, where $\sigma$ is 
a certain cyclic shift $j\mapsto j+r$, $r\in\{1,\dots,k\}$.
By the AM-GM inequality in the form \eqref{AMGMperm}, we get $S_{n,k,-1}(\xx)\geq n$.
\end{proof}

\begin{corollary}
For any $k$ and any $p\leq-1$ we have $B_{k,p}=1$.
In particular, it is so in the limiting case $p=-\infty$,
where $M_{k,-\infty}(\xx)=\min x_j$.
\end{corollary}

%------------------------------------------------------------------------------------------------
\iffalse
\begin{theorem}
\label{thm:p0}
For any $k\geq 1$, the equality $B_{k,0}=1$ holds true. 
\end{theorem}

\begin{proof}
In the definition of $B_{k,0}$ the $0$-th order means are the geometric means: $M_{k,0}(x_{j+1},\dots,x_{j+k})=(x_{j+1}\dots x_{j+k})^{1/k}$.
Therefore
$$
 S_{n,k,0}(\xx)=\sum_{j=1}^n y_j,
$$
where
$$
y_j=\prod_{i=1}^k\left(\frac{x_j}{x_{j+i}}\right)^{1/k}.
$$
Clearly, $\prod_{j=1}^n y_j=1$, so by the AM-GM inequality $ S_{n,k,0}(\xx)\geq n$.
\end{proof}

\begin{corollary}
For any $k$ and any $p\leq 0$ we have $B_{k,p}=1$.
In particular, it is so in the limiting case $p=-\infty$,
where $M_{k,-\infty}(\xx)=\min x_j$.
\end{corollary}
\fi
%----------------------------------------------------------------------------------------------------

\begin{remark}
For $p=0$, in the definition of $S_{n,k,0}$ the $0$-th order means are the geometric means: $M_{k,0}(x_{j+1},\dots,x_{j+k})=(x_{j+1}\dots x_{j+k})^{1/k}$. 
The reader can easily verify, that the AM-GM inequality implies the equality $B_{k,0}=1$ and therefore $B_{k,p}=1$ for all $p\leq 0$.   The optimal bounds $\nu_k=\sup\{p\mid B_{k,p}=1\}$ are unknown.
Diananda \cite{Diananda_1974} proved that $\nu_2\geq(\sqrt{5}-1)/2$.
\end{remark}

There is another  fully tractable and more interesting
case: $p=+\infty$, where $M_{k,+\infty}(\xx)=\max x_j$.
Here we know {\em a priori}\ that $B_{k,+\infty}<1$,
since $B_{k,+\infty}\leq B_{k,1}$ and $B_{k,1}<1$
by Theorem~\ref{thm:diananda}. Recall that $\lfloor\cdot\rfloor$ denotes the ``floor'' function
(the greatest integer lower bound for the given number).

\begin{theorem}
\label{thm:cyclic-infty}
The formula
$$
 A_{n,k,+\infty}=\left\lfloor\frac{n+k-1}{k}\right\rfloor
$$
holds true. Consequently, 
$$
 B_{k,+\infty}=1/k.
$$
\end{theorem}

This theorem is proved in \cite{Diananda_1973} by reduction to the AM-GM inequality in a judicious way. It also follows from 
the more general theory \cite[\S~6]{S21b}, to which we soon turn.

As a corollary we get a partial result in the direction (ii) from the above ``to do'' list.
\begin{theorem}
\label{thm:cyclic-p}
For any $p\in(1,\infty)$ the inequality
\begin{equation}
\label{Bkp_lb}
 B_{k,p}\geq B_{k,1}^{1/p} \,
\frac{k^{-1/q^{2}}}{k-1} U_{1/p}(k) U_{1/q}(k)
\end{equation}
holds true, where $q$ is the conjugate to $p$ exponent: $1/q+1/p=1$ and 
$$
 U_t(k)=%\left(\frac{k^{t}-1}{t}\right)^{t}.
(k^{t}-1)^t\,{t}^{-t}.
$$
\end{theorem}

\begin{proof}
%Let us fix the $n$-tuple $\xx$ and 
Put $y_{j,k,p}=M_{k,p}(x_{j+1},\dots,x_{j+k})$.
Denote $u_j=(x_j/y_{j,k,1})^{1/p}$ and $v_j=(x_j/y_{j,k,\infty})^{1/q}$. 
Then $S_{n,k,p}(\xx)=\sum_{j=1}^n u_j v_j =(u,v)$,
while $\|u\|_{p}^p %=\sum_{j=1}^n u_j^p
=S_{n,k,1}(\xx)$
and $\|v\|_q^q %=\sum_{j=1}^n v_j^q
=S_{n,k,\infty}(\xx)$.

By the convexity of means \cite[\S~2.9]{HardyLittlewoodPolya1934}, we have $y_{j,k,p}\leq y_{j,k,1}^{1/p} y_{j,k,\infty}^{1/q}$.

Also we have $1\leq u_j^p/v_j^q\leq k$ for all $j$.
By the reversed H\"{o}lder inequality due to Diaz, Goldman and Metcalf (see \cite{Nehari_1968}, \cite[Ch.~V, \S~13]{Mitrinovic_1993}), $\|u\|_p\|v\|_q\leq C_{p,q}(k) \cdot(u,v)$.
Substituting the expression for $C_{p,q}(k)$ from the cited references and the estimates $\|u\|_p\geq B_{k,1}^{1/p}$ and $\|v\|_q\geq k^{-1/q}$ (by Theorem~\ref{thm:cyclic-infty}), we come to the claimed result.
\end{proof}

\begin{remark}
The factor $k^{-1/q^2}/(k-1) U_{1/p}(k)U_{1/q}(k)$ in \eqref{Bkp_lb} decreases from $1$ to $1/k$ as $p$ increases from $1$ to $\infty$ for any $k>1$.
\end{remark}

\section{Graphic sums}    
\label{sec:gsums}

As explained in the previous section, the problem of 
determining the constant $B_{k,p}$ is completely trivial in the case $p=-\infty$ (as a consequence of the estimate for $p=-1$) and tractable in a rather elementary way in the case $p=+\infty$. In these two limiting cases there is no need of any calculus beyond the AM-GM inequality: no critical points, no Lagrange multipliers.    

It seems therefore plausible that meaningful lower bounds for sums with minimum or maximum of some subsets of indeterminates in the denominators can be obtained in
a more general situation. These considerations prompted me to introduce the notion of ``graphic'' $p$-sums as a generalization of cyclic $p$-sums and to explore the cases $p=\pm\infty$ in depth.

\begin{definition}
Let $\Gamma$ be a directed graph, $V$ be the set of its vertices, and $\Gamma^+(v)$ denote the out-neighborhood of a vertex $v\in V$. Let $\mathbf{x}$ be a vector with components labeled by $v\in V$. 
The graphic $p$-sum corresponding to the graph $\Gamma$
is the function
$$
 S_p(\mathbf{x}|\Gamma)=\sum_{v\in V} \frac{x_v}{M_p(\xx|\Gamma^+(v))},
$$
where, for the given set $\Omega\subset V$, 
$$
M_p(\xx|\Omega)=\left(\sum_{v'\in\Omega} x_{v'}^p\right)^{1/p}.
$$
In particular, we define the ``graphic min-sums'' corresponding to the case $p=-\infty$ 
$$     
S_{\land}(\xx|\Gamma)=\sum_{v\in V} \frac{x_v}{\min_{v'\in\Gamma^+(v)} x_{v'}}.
%     \label{SGmin}
$$
and the ``graphic min-sums'' corresponding to the case $p=+\infty$ 
$$
   S_{\lor}(\xx|\Gamma)=\sum_{v\in V}
   \frac{x_v}{\max_{v'\in\Gamma^+(v)} x_{v'}}.
 %\label{SGmax}
$$
\end{definition}

The Shapiro-Diananda sum $S_{n,k,p}$ corresponds to the 
graph $\Gamma$ with cyclic symmetry: the directed edges beginning at the vertex $i$ have ends $i+1,\dots,i+k$.

\section{Graphic min-sums}
\label{sec:gminsum}

Unlike in Section~\ref{sec:psums}, the lower bounds
for graphic min-sums are in general nontrivial.

Let us fix the graph $\Gamma$. For any given vector
$\xx$ we have
\begin{equation}
\label{Sratio}
 S_\land(\xx|\Gamma)=\sum_{v\in V}\frac{x_v}{x_\mu(v)},
\end{equation}
where $\mu(v)\in\Gamma_{\min}^+(v|\xx)$
and $\Gamma^+_{\min}(v|\xx)$ denotes the set of indices 
that minimize $x_w$ over $w\in \Gamma^+(v)$.
If all the values $x_v$ are distinct (and only then), then $\Gamma^+_{\min}(v)$ is a one-element set for every $v$
and the function $v\mapsto \mu(v)$ is defined uniquely.
%However, generally speaking, $\mu(\cdot)$ is a {\em choice function}: for every $v$ there is a set of available values and it must select one.

Let me emphasize that the function $\mu(v)$ depends on the vector $\xx$ and possibly (if not all $x_i$ are distinct) also on a finite choice. Yet the number of all possible functions $\mu(\cdot)$ is finite (the upper bound is $n^n$, where $n=|V|$). So in principle the minimizing solution can be found as follows. (Or so it seems\dots.)

Take any function $\mu:\;V\mapsto V$ such that
$\mu(v)\in\Gamma^+(v)$ for all $v\in V$. 
For this function, find the minimum of the sum in the right-hand side of \eqref{Sratio}. For any minimizer $\xx$, check whether the consistency condition
$\mu(v)\in\Gamma^+_{\min}(v|\xx)$ is met for all $v$. If not, reject $\xx$. 

Do it for all admissible functions $\mu(\cdot)$. 
Compare the values of the objective function
corresponding to the minimizers that were not rejected.
Find the minimum; that's it.

The number of functions $\mu(\cdot)$ is huge, but for every single one the minimization looks easy, basically amounting to the AM-GM, taking into account that not all of the $v$'s appear as indices of the variables in the denominators in \eqref{Sratio}. In any case, the algorithm assures us that the answer will be an integer.

But \dots we have Problem~\ref{prob:minsum} corresponding to the graph $\Gamma$ on three vertices, $V=\{1,2,3\}$, with edges $1\to 3$, $2\to 3$, $3\to 1$, and $3\to 2$. And the answer is irrational!

Can you spot a flaw in the algorithm? Does it guarantee
that after all the required rejections anything is left?
(I highly recommend to work out Problem~\ref{prob:minsum} fully to get the idea.)

\iffalse
In fact, once we have chosen the function $\mu(\cdot)$,
the variables $x_v$ in the right-hand side of \eqref{Sratio} are no longer independent. Say, in Problem~\ref{prob:minsum} we must set $\mu(1)=\mu(2)=3$
and we may choose $\mu(3)=1$. We've got the function
$$
\frac{x_1}{x_3}+\frac{x_2}{x_3}+\frac{x_3}{x_1}
$$
to minimize. But, since $\min(x_2,x_1)=x_1$ --- unless
we will have to reject $\xx$ --- the alleged minimizer 
$(0,1,1)$ does not work.
\fi

In fact, the proposed algorithm becomes valid after a  fix. We must organize a loop over combinatorial
structures other than functions $V\to V$, called {\em preferrential arrangements}
on $V$, to put in one block variables that must have equal values in a candidate minimizer. Then we put new variables, $y_B$, in correspondence to each block.
We refer to \cite[\S~7]{S21b} for details. 
The equation \eqref{Sratio} is rewritten as
\begin{equation}
\label{Bratio}
 S_\land(\xx|\Gamma)=\sum_{B} |B|\frac{y_B}{y_{\nu(B)}}
\end{equation}
with an appropriate function $\nu(\cdot)$ 
that replaces
$\mu(\cdot)$; $\nu$ maps the set of blocks into itself.
In Problem~\ref{prob:minsum} this general scheme yields the unconstrained optimization problem
$2y_I/y_{II}+y_{II}/y_I\to\min$ with blocks being $I=\{1,2\}$ and $II=\{3\}$.

Problem~\ref{prob:minsum} and the above general discussion are concerned with minimization of the min-sum for an individual graph.

Problem~\ref{prob:extrminsum} is of a somewhat different nature. You have to estimate the lower bound of a graphic min-sum, but the graph of friendship $\Gamma$ is not specified. The only constraints are: the size of the vertex set $|V|=40$ and (importantly) the property of strong connectendess of $\Gamma$ (worded in terms of the chains of friends). The optimization space
here is the Cartesian product of $\mathbb{R}^n_+$ (values of vectors $\xx$) and the set of all strongly connected directed graphs of $n$ vertices. 
It is an instance of extremal graph problems.

The following inequality is a part of Theorem~4 in \cite{S21b}.

\begin{theorem}
\label{thm:extrmin}
For any strongly connected graph $\Gamma$ on $n$ vertices the minimum value of the min-sum determined by $\Gamma$ satisfies the estimate
$$
 \min_{\xx}S_{\land}(\xx|\Gamma)> e\ln(n+1-\ln(n+1)).
$$
\end{theorem}

In the proof, the extremal graph $\Gamma_*$ is explicitly identified and it is shown that
$$
\min_\xx S_{\land}(\xx|\Gamma_*)=\min_{1\leq k\leq n-2}
\min_{x,y}
\left(k\left(\frac{x}{y}\right)^{1/k}+(n-k)\frac{y}{x}\right),
$$
so that
$$
 \ln \min_\xx S_{\land}(\xx|\Gamma_*)=\min_{1\leq k\leq n-2}\left(\ln(k+1)+\frac{\ln(n-k)}{k+1}\right).
$$ 
My study of the described extremal graph problem and
its reduction to the minimization problem for the explicit elementary function motivated the paper \cite{KalachevSadov_2018}, where a proof of the inequality --- first conjectured numerically ---
$$
 \ln\ln(r-\ln r)<\min_{0\leq x\leq r-1}
\left(\ln x+\frac{\ln(r-x)}{x}\right),
$$
and of a tight upper bound for the same function were given. (Thanks to my coathor's %G.\,V.~Kalachev
perseverance, the published proof is purely analytical with no reference to numerical evaluations.)

\smallskip
Theorem~\ref{thm:extrmin} implies: 
since $e\ln(41-\ln 41)\approx 9.836> 9.8$, the answer in Problem~\ref{prob:extrminsum} is {\bf no}.

\section{Graphic max-sums and girth}
\label{sec:gmaxsum}

The minimization problem for a graphic max-sum leads 
to a well-known combinatorial quantity.

Recall that {\em girth}\ of a directed graph $\Gamma$,
denoted $g(\Gamma)$, is the length of its shortest loop.

We will also need the notion of {\em strong reduction}.

Introduce the pre-order relation $\leqslant $ on the set $V$ of vertices of a directed graph $\Gamma$: it is the transitive closure of the adjacency relation $v\to v'$.
That is, $v\leqslant v'$ if there exists a directed path from $v$ to $v'$. 
If $v\leqslant v'$ and $v'\leqslant v$, we say that $v$ and $v'$ are equivalent ($v\sim v'$). The set $\hat V=V/\sim$ of equivalence classes is partially ordered by the same relation $\leqslant$.

\begin{definition}
\label{def:strongreduction}
The equivalent classes corresponding to the maximal elements of the partially ordered set $\hat V$
are called the {\em final strong components}\ of the graph $G$. 
\end{definition} 

Informally speaking, final strong components are black holes of the universe that is our graph: from anywhere you can fall into (at least) one of them, and once you are there
--- no way out. 

\begin{theorem}
\label{thm:gmaxsum}
If $\Gamma$ is a strongly connected graph, then
$$
 \inf_{\xx}S_\lor(\xx|\Gamma)=g(\Gamma).
$$
In general, the greatest lower bound of the graphic max-sum for $\Gamma$
is the sum of girths of its final strong components.

In particular, $\min S_\lor(\xx|\Gamma)$ is always an integer.
\end{theorem}

This theorem is proved in \cite[\S~3]{S21b}. It is simple.
You will know the mechanics of the proof and understand the appearance of girth, as well as the reason why `$\inf$' rather that `$\min$' is used in the formulation, once you work out Problem~\ref{prob:maxsum}. (AM-GM is relevant again!)

\smallskip
Now it is easy to give an answer in Problem~\ref{prob:extrmaxsum}: {\bf yes}, this can happen. (And, once the value $4.89$ has occured in some round of the game, a value $4+\epsilon$
with arbitrarily small positive $\epsilon$ can be obtained if the participants persist.) Moreover,
the value $2+\epsilon$ can happen: just a
pair of ``symmetric'' friends (that is, friends in the normal old sense) is needed. 

\smallskip
While the analytical side of the max-sum minimization 
is simple, the answer, in terms of girth(s), requires 
evaluation by some combinatorial algorithm.

\smallskip
Extremal problems of graph theory related to girth are difficult. 

Let $\mathfrak{G}_s(n,k)$ be the set of strongly connected directed graphs on $n$ vertices with minimum outdegree $k$. (That is, $\min_{v\in V}|\Gamma^+(v)|=k$.)

The famous {\em Caccetta-H\"aggkvist conjecture}\ (CHC)
\cite[Conjecture~8.4.1]{BangGutin_2002} claims that
$$
 g(\Gamma)\leq \lceil n/k\rceil
\quad\text{for any $\Gamma\in \mathfrak{G}_s(n,k)$}.
$$ 

This conjecture can be illustrated with our Problem~\ref{prob:extrmaxsum} if a quantifier is changed. Let us ask, whether it is {\em guaranteed} that
in the game as defined a value $< 4.9$ can be obtained
(with a suitable choice of the participants' numbers).

In the theoretical formalism the question becomes: whether for {\em any}\ graph $\Gamma\in\mathfrak{G}_s(40,12)$ the inequality
$\inf_{\xx}S_\lor(\xx|\Gamma)<4.9$ is true?

In view of Theorem~\ref{thm:gmaxsum}, it is equivalent to the question: whether for any $\Gamma\in\mathfrak{G}_s(40,12)$ the inequality
$g(\Gamma)\leq 4$ is true?

Assuming the validity of CHC, since $\lceil 40/12\rceil=4$,  we see that the answer is positive. 

I do not know the status of CHC for the pair $(40,12)$. Is it computationally feasible to verify it for these values?

\section{``Monomial quotient'' graphic sums and Shallit's problem}    
\label{sec:prodconstraints} 

The next topical thread begins with Eq.~\eqref{Bratio}.
If the function $\nu(\cdot)$ in its right-hand side were a permutation, then we would be dealing just with a weighted version of the AM-GM problem, as is the case in 
Problem~\ref{prob:minsum}. However, in general the sums \eqref{Bratio} that appear in the course of
minimization of $S_\land(\xx|\Gamma)$ may 
have more complex combinatorial structure.
In particular, a variable $y_B$ can appear in more than one numerator. 

The relevant formalism involves directed graphs with
weighted edges. 

Let $\Gamma$ be a directed graph on the vertex set $V$
and let $E$ be the set of its (directed) edges. 
To every edge $u\in E$ there is assigned a positive real number $w_u$. We treat these numbers as the components
of the vector $\mathbf{w}$ of dimension $|E|$.

If $u$ is an arrow from the vertex $v$ to the vertex $v'$, we write $v=\alpha(u)$, $v'=\beta(u)$.
%(In words, $\alpha(\cdot)$ and $\beta(\cdot)$ are, respectively, the `start' and `finish' functions.)

Let $\yy$ be a vector with components indexed by the set $V$. 

Introduce a graphic sum of ``monomial quotient'' type:
$$
 S_\Gamma^\div(\ww|\yy)=\sum_{u\in E}w_u \frac{y_{\alpha(u)}}{y_{\beta(u)}}
$$

The minimization problem: to find
$$
 f^\div_\Gamma(\ww)=\inf_{\yy} S_\Gamma^\div(\ww|\yy)
$$ 
--- is a basic step in the computation of minima of min-sum discussed in Section~\ref{sec:gminsum}.

If $w_u=1$ for all $u$, we write $S_\Gamma^\div(\ww|\yy)=S_\Gamma(\yy)$
and $f^\div_\Gamma(\ww)=f^\div_\Gamma$
to simplify the notation.

This new minimization problem can be stated as a constrained optimization problem as follows. 

Put $z_u=y_{\alpha(u)}/u_{\beta(u)}$. 
We have to minimize the dot-product 
$$
\langle\ww,\zz\rangle=\sum_{u\in E} w_u z_u
$$
subject to constraints
$P_L(\zz)=1$, where $L$ goes over some {\em basis of cycles}\ of the graph $\Gamma$. 
Here $L$ is a cycle, that is, a closed path consisting of the edges $u_1,\dots,u_\ell$, say, and
$$
 P_L(\zz)=\prod_{i=1}^\ell z_i.
$$ 
 
Thus, the minimization problem for graphic min-sums motivated a systematic study of the minimization problem for monomial graphic quotients. However the latter turned
out to have a nontrivial life of its own. 

The AM-GM problem obviously belongs to this type, the graph $\Gamma$ being just the $n$-cycle. But this is not too impressing.
A much more remarkable example is the function in  Problem~\ref{prob:Shallit}, which for the purpose of this section is better to write in the homogeneous form:
$$
 \sum_{j=1}^n \left(\frac{x_j}{x_0}+\frac{x_0}{x_j}\right)+\sum_{j=1}^{n-1}\frac{x_j}{x_{j+1}}.
$$
If the second sum on the right contained one more term $x_n/x_1$ making it cyclic,
then by the AM-GM the minimum would be attained at $\xx=(1,\dots,1)$ and would be equal to $3n$. 
We also see that $\min g_n(\xx)\leq 3n-1$. 
%in the problem as it stands. 
Thus, if the constant $C$ exists as is claimed, then $C\geq 1$.

Problem~\ref{prob:Shallit} is equivalent to 
the one proposed by J.~Shallit in 1994 \cite{Shallit_1994}, apparently being a summary of a numerical experiment. The value $C\approx 1.3694514$
was given by Shallit already.
The equivalent form was given in \cite{Shallit_1995} as the first step of their solution. However, the  purported solution, while 
containing insightful observations, was not analytically rigorous and recoursed to numerics for critical facts concerning the behaviour of auxiliary sequences. 
In \cite{S21a} I revisited Shallit's problem and filled in the analytical gaps of \cite{Shallit_1995}. The rigorous solution produced an exponentially small estimate for the remainder $o(1)$ in the asymptotic formula as a bonus. 

Let us see how Shallit's problem fits in this section's scheme. Consider the graph $\Gamma$ with $n+1$ vertices
labeled $0,1,\dots,n$ and $3n-1$ edges of three sorts:
(i) $n$ edges $u_{j0}=(j\to 0)$ corresponding to the terms $x_j/x_0$; (ii) $n$ edges $u_{0j}=(0\to j)$ corresponding to the terms $x_0/x_j$;
(iii) $n-1$ edges $u_{j,j+1}$ 
corresponding to the terms $x_j/x_{j+1}$. The basis of cycles for the graph $\Gamma$ contains $2n$ elements:
$n$ cycles of length $2$ of the form $\{u_{0j},u_{j0}\}$
and $n$ cycles of length $3$ of the form $\{u_{0j},u{j,j+1},u_{j+1,0}\}$.
We have $\min_{\xx} g_n(\xx)=f^\div_\Gamma$.

The graphic sum interpretation does not yield a magic solution to Shallit's problem. The actual solution amounts to a detailed analysis of the critical point equations written in the form of a recurrence relation for the components of the minimizer. Yet one aspect of the 
``graphic quotient sums'' theory can be appreciated:
the uniqueness of minimizer.

\begin{theorem}
\label{thm:unique}
For any strongly connected directed graph $\Gamma$ and any set of positive edge weights $w_u$ a positive minimizer $\yy$ for the optimization problem $S_\Gamma(\ww|\yy) \to  \min$
exists. It is unique up to a positive multiple.
\end{theorem}

This theorem is a combination Problems~15 and 40A in \cite{S20}. 
Once you've got an optimization problem presented in the form as described, you know right away that the minimizer is unique.

In particular, a family of sums generalizing Shallit's
one is proposed in \cite[\S~7]{S20}. Theorem~\ref{thm:unique} spares us of
the daunting task need to verify the uniqueness of a minimizer for each of them indivivually.

\section{The Mavlo-Georgiev inequality and symmetry of the minimizer}    
\label{sec:Mavlo}

In this section I will demonstrate a curious unexpected ``application'' of Theorem~\ref{thm:unique}, which is of a psychological as much as of technical nature.

The following is a simple corollary of Theorem~\ref{thm:unique}.

\begin{theorem} \cite[Problem~19]{S20}
\label{thm:symmetry}
If $\sigma$ is an automorphism of the given directed,
strongly connected, edge-weighted graph $\Gamma$ (we consider $\sigma$ as a permutation of the set $V$), then the minimizer for the problem $S_\Gamma(\ww|\yy) \to  \min$
is $\sigma$-invariant, that is, $y_v=y_{\sigma(v)}$
for any $v\in V$.
\end{theorem}

Let us now turn to Problem~\ref{prob:Mavlo}. It is a homogeneous form of the inequality
\begin{equation}
\label{Mavlo-Georgiev}
 \frac{1}{u(1+v)}+\frac{1}{v(1+w)}+\frac{1}{w(1+u)}\geq
\frac{3}{1+uvw}.
\end{equation}
(Put $x=(uvw)^{1/3}$ and let $u=xb/a$, $v=xc/b$, $w=xa/c$.)

%The Nesbitt inequality
%is actually {\em symmetric}\ in the sense that its left-hand side is invariant under any permutation of the letters $a,b,c$. Introduction of the parameter $x$ as in
%Problem~\ref{prob:Nesbitt}(b) reduces the full symmetry group $S(3)$ to the cyclic group $Z(3)$ generated by the shift $(a,b,c)\mapsto(b,c,a)$. I learned about this problem, in another guise, from  An equivalent (and yet shorter) solution will be given in Section~\ref{sec:prodconstraints} below,
%along with a refinement.    

I learned about the latter inequality from Alex Bogomolny's remarkable ``Cut-the-knot'' website \cite{cuttheknotineq}, where the story and a reference can be found. I will call it the Mavlo-Georgiev by the names of the proposer and the author of a short solution.

\begin{proof}[Proof of the Mavlo-Georgiev inequality]
The idea is to replace ``difficult'' denominators by ``easy'' ones, to get a sum in the form similar to that 
in the proof of Theorem~\ref{thm:p-1}.

Let us express $a$, $b$ and $c$ in terms of $A=b+cx$, $B=c+ax$ and $C=a+bx$. We have
$c-ux^2=B-Cx$, so $Ax^2-Cx+B=c(1+x^3)$, and the solution
is
\begin{align}
& \tilde a=Bx^2-Ax+C,
\nonumber
\\
& \tilde b=Cx^2-Bx+A,
\nonumber
\\
&
\tilde c=Ax^2-Cx+B,
\nonumber
\end{align} 
where $\tilde a=a(1+x^3)$ etc.

Denote the left-hand side of the inequality we are proving
by $L$. Then
\begin{align}
(1+x^3)L &=\frac{Bx^2-Ax+C}{A}+
\frac{Cx^2-Bx+A}{B}+
\frac{Ax^2-Cx+B}{C}
\nonumber
\\[1ex]
& =\frac{B}{A}x^2+\frac{A}{B}
+\frac{C}{B}x^2+\frac{B}{C}
+\frac{A}{C}x^2+\frac{C}{A}-3x
\label{Gsum-Georgiev}
\\[1ex]
&\geq 3\cdot 2\sqrt{x^2}-3x=3x,
\nonumber
\end{align}
by the AM-GM inequality. Hence $L\leq 3x/(1+x^3)$.
\end{proof}

On a close look this is but a reformulation of Georgiev's proof given in \cite{cuttheknotineq}. Let us now observe that the expression \eqref{Gsum-Georgiev},
if we ignore the ``constant term'' $-3x$ (independent of the optimization variables: we treat $x$ as a parameter),
is a monomial quotient sum $S_{\Gamma}(\ww|\yy)$
corresponding to the complete directed graph on three vertices $A$, $B$, $C$, with edge weights 
$w_{BA}=w_{CB}=w_{AC}=x^2$ and $w_{AB}=w_{BC}=w_{CA}=1$.

Knowing Theorem~\ref{thm:unique} and its consequence Theorem~\ref{thm:symmetry}, I know at once that the minimizer must by cyclically symmetric, that is, proportional to $[1,1,1]$. Then I know without calculations that the minimum is attained at $A=B=C=1$
and is equal to $3/(1+x)$. This is a strict improvement of Mavlo-Georgiev's estimate unless $x=1$,
because $(1+x^3)/x-(x+1)=(x-1)^2(x+1)$. 
Summarizing and translating the result into the variables 
$u,v,w$ as in \eqref{Mavlo-Georgiev}, we state

\begin{theorem}[A sharp form of the Mavlo-Georgiev inequality]
For any positive values of the parameters, the inequalities
$$
\frac{a}{b+cx}+\frac{b}{c+ax}+\frac{c}{a+bx}\geq\frac{3}{1+x}
$$
and
$$
 \frac{1}{u(1+v)}+\frac{1}{v(1+w)}+\frac{1}{w(1+u)}\geq
\frac{3}{(uvw)^{1/3}\left(1+(uvw)^{1/3}\right)}.
$$
hold true.
\end{theorem}

Why was it easy to overlook this improvement?
Of course, looking at the expression \eqref{Gsum-Georgiev}
a temptation to apply the AM-GM and be done immediately is very strong. On the other hand, equally strong in such elementary problems is the suspicion that 
the symmetry of problem is inherited by the minimizer.
It is not always easy to prove. Here, too, if you write
the critical point equations, they are quite a mess.
Given Problem~\ref{prob:Mavlo} to solve and seeing that
the right-hand side is not equal to the value 
of the left-hand side at the symmetric tuple, one can even suspect (did you?) that the actual critical points are not symmetric vectors.

\smallskip
The method of ``trivialization of denominators'' used in the reduction of the Mavlo-Georgiev objective function to a grapic monomial quotient sum looks promising enough to deal with Shapiro's sums. After all, Problem~\ref{prob:Nesbitt} is a particular case of Problem~\ref{prob:Mavlo} and at the same time the case
$n=3$ of Shapiro's inequality\dots.

Indeed, for any odd $n$, the system of equations
$x_{i+1}+x_{i+2}=A_i$, $i=1,\dots,n$ (with addition of indices modulo $n$) is uniquely solvable for $\xx$.
(It is not so if $n$ is even. For example, if $n=4$,
we have the necessary condition for consistency of the system: $A_1+A_3=A_2+A_4$, because both sums must be equal to $x_1+x_2+x_3+x_4$).

What prevents us from passing from variables $x$ to variables $A$, getting the optimization problem for a
monomial quotient sum, etc.?

Answer: nothing. The obstacle lies with ``etc''. In the
resulting graph not all edge weights will be positive
and the wonderful strength of Theorem~\ref{thm:unique}
will be lost. Even if the minimizer is symmetric as in the cases of odd $n\leq 23$, this fact does not follow by Theorem~\ref{thm:symmetry}. 

You can check that in the case $n=5$ the components
of the vector $\xx$ are expressed as
$x_1=(-A_1+A_2-A_3+A_4+A_5)/2$ and the pattern repeats cyclically. The resulting graph is the complete directed graph on 5 vertices. The quotient $x_1/A_1$ contains the
nonconstant monomial quotient $-A_3/(2A_1)$ with negative coefficient. In our graph the negative edge weight $u_{3,1}=-1/2$ shows up.  No roses without thorns, alas!

\section{Diananda-type sums with variable lengths of \\ denominators} 
\label{sec:maxforward}
Now I will discuss a generalization of the Diananda sums in a different direction, unrelated to graphic sums.

Theorem~\ref{thm:diananda} implies that $A_{n,k}\geq n\ln 2$. That is, for all $n\geq k\geq $ and all positive $n$-dimensional vectors $\xx$ the cyclic inequality
$$
 \sum_{j=1}^n\frac{x_j}{\frac{1}{k}(x_{j+1}+\dots+x_{j+k})}\geq n\ln 2
$$
is true. 
One is tempted to let $k$ in the right-hand side of (\ref{Snkp}) vary with $j$.

Suppose a vector $\kk$ of dimension $n$ is given whose components are natural numbers in the range $[1,\dots, n]$. Introduce the sums
$$
 S_{n,\kk}(\xx)=\sum_{j=1}^n\frac{x_j}{\frac{1}{k_j}(x_{j+1}+\dots+x_{j+k_j})}.
$$
What can be said about the best lower bound for $S_{n,\kk}$ uniform with respect to $\kk$. Is it still
linear in $n$? Is it at all unbounded as $n\to\infty$?
The next theorem gives an answer.

\begin{theorem}
\label{thm:maxforward}
Denote
$$
 A_{n,*}=\sup_{\kk}\sup_{\xx} S_{n,\kk}(\xx),
$$
where $k$ goes over the set of all integer-valued $n$-vectors with components in the range $[1,\dots, n]$.
Then
\begin{equation}
\label{Anmax_asym}
A_{n,*}=e\ln n-A+O\left(\frac{1}{\ln n}\right)
\end{equation}
as $n\to\infty$. Here $A\approx 1.704656$ is a numerical constant.
\end{theorem}

This is the main result of \cite{S22a}. It will be illustrated with Problem~\ref{prob:maxforward}.
Indeed, the posed question is equivalent to: can $S_{n,\kk}(\xx)$ ever be less than 22 if $n=2022$.
Put another way, is it true that $A_{2022,*}<22$?

Since $e \ln(2022)-A\approx 18.99$, we can rather confidently answer in the affirmative. The difference
$22-18.99$ seems big enough to accommodate the asymptotic remainder.

%But now the best lower bound for $S_{n,*,1}$ is no longer linear in $n$. 

\section{A perturbed functional equation for the \\ AM-GM problem}

As a matter of fact, in \cite{S22a} I did not prove
Theorem~\ref{thm:maxforward} in full but showed that
$A_{n,*}=F(n)$, where $F(\cdot)$ is a certain function
whose asymptotic behaviour I examined in detail
in a separate paper \cite{S22b}. I was able to find the main part of the remainder $O(1/\ln x)$ 
corresponding to $O(1/\ln n)$ in \eqref{Anmax_asym}.

The function $F(x)$ can be defined by a functional equation that looks similar to \eqref{AMGMfeq}:
\begin{equation}
\label{F}
 F(x)=
 \min_{0<y<x-1} \left(F(y)+\frac{x}{y+1}\right),
\quad x>1,
\end{equation}
with initial condition $F(x)=x$ for $0\leq x\leq 1$.

It is also shown in \cite{S22b} that
$$
 F(x)=\inf_{n\in\mathbb{N}} F_n(x),
$$
where
\begin{equation}
\label{Fn}
 F_n(x)=\inf_{t_1,\dots,t_{n-1}>0} \left(t_1+\frac{t_2}{t_1+1}+\dots+\frac{t_{n-1}}{t_{n-2}+1}+\frac{x}{t_{n-1}+1}\right).
\end{equation}

\begin{theorem}
\label{thm:F1asym}
Let $\|y\|$ denote the distance from the number $y$ to the nearest integer. The asymptotic behaviour of the function $F(x)$ is as follows:
$$
 F(x)=e\ln x-A+\frac{e\|b+\ln x\|^2}{2\ln x}+O\left(\frac{1}{(\ln x)^{2}}\right),
$$
where $A$ is the same constant as in Theorem~\ref{thm:maxforward} and $b\approx 0.69739$. 
\end{theorem}

The proof of Theorem~\ref{thm:F1asym} is rather long,
comparable in some ways with the proof of Shallit's conjecture (see Section~\ref{sec:prodconstraints}). In both cases the minimizer is treated
as a trajectory of a discrete dynamical system. 

To get an idea of the proof, specifically, a construction
of the minimizer clear of technical nuances concerning monotonicity, convexity, etc., which the overall proof has to deal with, you can try Problem~\ref{prob:modifiedAMGM}. (It is not as easy as it may look at first sight but doable. The minimizer is $[1,\,4,\,25/2]$.) 

Much simpler than Theorem~\ref{thm:F1asym} is its analog \cite[Theorem~2]{S22b} concerning the function $f(x)$ related to the 
the AM-GM. Recall Eq.~\eqref{AMGMfeq}, a prototype of
\eqref{F}, and define
$$
 f(x)=\inf_n f_{n\in\mathbb{N}}(x)=\inf_{n\in\mathbb{N}}
nx^{1/n}.
$$

\begin{theorem}
\label{thm:F0asym}
The following asymptotic formula holds:
$$
 f(x)=e\ln x+\frac{e\|\ln x\|^2}{2\ln x}+O\left(\frac{1}{(\ln x)^{2}}\right).
$$
\end{theorem}

I do not know whether this result is new; in any case it is far less known than many other facts about and around the AM-GM.

%\smallskip

\section{Conclusion}
The variable in a cyclic sum that was seen in the numerator of the very first term 
is found in the denominator at the end of the loop.

Likewise, our tour has completed a full circle and returned to the AM-GM, featuring another side of it.

\end{document}